\documentclass[oneside,english, a4paper]{amsart}
\usepackage[T1]{fontenc}
\usepackage[latin9]{inputenc}
\usepackage{amstext}
\usepackage{amsthm}
\usepackage{amssymb}

\makeatletter
\numberwithin{equation}{section}
\numberwithin{figure}{section}
\theoremstyle{plain}
\newtheorem{thm}{\protect\theoremname}[section]
  \theoremstyle{plain}
  \newtheorem*{question*}{\protect\questionname}
  \theoremstyle{definition}
  \newtheorem{defn}[thm]{\protect\definitionname}
  \theoremstyle{plain}
  \newtheorem{lem}[thm]{\protect\lemmaname}
  \theoremstyle{remark}
  \newtheorem{rem}[thm]{\protect\remarkname}
  \theoremstyle{plain}
  \newtheorem{prop}[thm]{\protect\propositionname}
  \theoremstyle{plain}
  \newtheorem{cor}[thm]{\protect\corollaryname}
  \theoremstyle{definition}
  \newtheorem{example}[thm]{\protect\examplename}
  \theoremstyle{remark}
  \newtheorem*{acknowledgement*}{\protect\acknowledgementname}

\usepackage{hyperref}
\theoremstyle{plain}

\newtheorem*{thmaux}{Theorem \theoremauxnum}
\newtheorem*{coraux}{Corollary \theoremauxnum}
\newtheorem*{propaux}{Proposition \theoremauxnum}

\gdef\theoremauxnum{1}

\newenvironment{thmff}[2]{%
  \def\theoremauxnum{\ref{#2}}
  \begin{thmaux}[#1]
}{%
  \end{thmaux}
}

\newenvironment{corff}[2]{%
  \def\theoremauxnum{\ref{#2}}
  \begin{coraux}[#1]
}{%
  \end{coraux}
}

\allowdisplaybreaks[4]

\makeatother

\usepackage{babel}
  \providecommand{\acknowledgementname}{Acknowledgement}
  \providecommand{\corollaryname}{Corollary}
  \providecommand{\definitionname}{Definition}
  \providecommand{\examplename}{Example}
  \providecommand{\lemmaname}{Lemma}
  \providecommand{\propositionname}{Proposition}
  \providecommand{\questionname}{Question}
  \providecommand{\remarkname}{Remark}
\providecommand{\theoremname}{Theorem}

\newcommand{\ams}{ams}

\newcommand{\sep}{, }

\begin{document}
\global\long\def\norm#1{\left\Vert #1\right\Vert }

\global\long\def\abs#1{\left|#1\right|}

\global\long\def\set#1#2{\left\{  \vphantom{#1}\vphantom{#2}#1\right.\left|\ #2\vphantom{#1}\vphantom{#2}\right\}  }

\global\long\def\sphere#1{\mathbf{S}_{#1}}

\global\long\def\closedball#1{\mathbf{B}_{#1}}

\global\long\def\openball#1{\mathbb{B}_{#1}}

\global\long\def\duality#1#2{\left\langle \vphantom{#1}\vphantom{#2}#1\right.\left|\ #2\vphantom{#1}\vphantom{#2}\right\rangle }

\global\long\def\parenth#1{\left(#1\right)}

\global\long\def\curly#1{\left\{  #1\right\}  }

\global\long\def\blockbraces#1{\left[#1\right] }

\global\long\def\span{\textup{span}}

\global\long\def\image{\textup{Im}}

\textbf{}\global\long\def\support{\textup{supp}}

\global\long\def\N{\mathbb{N}}

\global\long\def\R{\mathbb{R}}

\global\long\def\Q{\mathbb{Q}}

\global\long\def\Rnonneg{\mathbb{R}_{\geq0}}

\global\long\def\C{\mathbb{C}}

\global\long\def\tocorr{\mathbb{\ \twoheadrightarrow\ }}

\global\long\def\restrict#1#2{#1|_{#2}}

\global\long\def\cal#1{\mathcal{#1}}

\global\long\def\ellinfty#1#2{\ell^{\infty}(#1,#2)}

\global\long\def\ellone#1#2{\ell^{1}(#1,#2)}

\global\long\def\c#1#2{\mathbf{c}(#1,#2)}

\global\long\def\czero#1#2{\mathbf{c}_{0}(#1,#2)}

\global\long\def\ellinftyomega#1{\ellinfty{\Omega}{#1}}

\global\long\def\elloneomega#1{\ellone{\Omega}{#1}}

\global\long\def\comega#1{\c{\Omega}{#1}}

\global\long\def\czeroomega#1{\czero{\Omega}{#1}}

\global\long\def\ellinftyN#1{\ellinfty{\N}{#1}}

\global\long\def\elloneN#1{\ellone{\N}{#1}}

\global\long\def\cN#1{\c{\N}{#1}}

\global\long\def\czeroN#1{\czero{\N}{#1}}

\global\long\def\conenorm#1{\left\llbracket #1\right\rrbracket }

\global\long\def\jamesseq{\mathcal{P}}

\global\long\def\finite#1{\mathcal{F}\parenth{#1}}

\global\long\def\selections#1{\mathcal{S}\parenth{#1}}

\global\long\def\ptwiselipat#1#2{\textup{PtLip}_{#1}\parenth{#2}}

\global\long\def\ptwiselip#1{\textup{PtLip}\parenth{#1}}

\global\long\def\ptwiselipconstant#1#2{\textup{\ensuremath{\Lambda}}_{#1}\parenth{#2}}

\newcommand{\Miek}{Miek Messerschmidt}

\newcommand{\MiekEmail}{mmesserschmidt@gmail.com}

\newcommand{\UPAddress}{Department of Mathematics and Applied Mathematics; University of Pretoria; Private~bag~X20 Hatfield; 0028 Pretoria; South Africa}

\newcommand{\UPAddresswithbreaks}{%
Department of Mathematics and Applied Mathematics\\
University of Pretoria\\
Private~bag~X20\\
Hatfield\\
0028 Pretoria\\
South Africa}

\newcommand{\ClaudeLeonAck}{The author's research was funded by The Claude Leon Foundation}

\newcommand{\paperabstract}{We prove that any correspondence (multi-function) mapping a metric
space into a Banach space that satisfies a certain pointwise Lipschitz
condition, always has a continuous selection that is pointwise Lipschitz
on a dense set of its domain. 

We apply our selection theorem to demonstrate a slight improvement
to a well-known version of the classical Bartle-Graves Theorem: Any
continuous linear surjection between infinite dimensional Banach spaces has a positively
homogeneous continuous right inverse that is pointwise Lipschitz on
a dense meager set of its domain.

An example devised by Aharoni and Lindenstrauss shows that our pointwise
Lipschitz selection theorem is in some sense optimal: It is impossible
to improve our pointwise Lipschitz selection theorem to one that yields
a selection that is pointwise Lipschitz on the whole of its domain
in general.}

\newcommand{\MSCCodesPrimary}{%
 	54C65\sep %  	Selections
	54C60 % 		Set-valued maps
}%
\newcommand{\MSCCodesSecondary}{% 	
	 	46B99 % Functional analysis   	None of the above, but in this section
}%
\newcommand{\paperkeywords}{Selection theorem\sep pointwise Lipschitz map\sep Bartle-Graves Theorem} %

\title{A pointwise Lipschitz selection theorem}
\author{\Miek}
\address{\Miek; \UPAddress}
\email{\MiekEmail}
\thanks{\ClaudeLeonAck}

\begin{abstract}
	\paperabstract
\end{abstract}
\subjclass[2010]{\MSCCodesPrimary (primary), and \MSCCodesSecondary (secondary)}
\maketitle

\newcommand{\lowerptlip}[1]{lower pointwise $#1$-Lipschitz}

\section{Introduction}

\global\long\def\preslawski{Przes\l awski}

Many correspondences (multi-functions) exhibit some form of Lipschitz
behavior. A classical example is what may be termed the inverse image
correspondence of a continuous linear surjection between Banach spaces:
Let $X$ and $Y$ be Banach spaces and $T:X\to Y$ a continuous linear
surjection. We define the \emph{inverse image correspondence} $\varphi:Y\to2^{X}$
by $\varphi(y):=T^{-1}\{y\}$ for all $y\in Y$. It is easily seen
that a map $\tau:Y\to X$ is a selection of $\varphi$ (meaning $\tau(y)\in\varphi(y)$
for all $y\in Y$) if and only if $\tau$ is a right inverse of $T$.
It is well-known, by the Bartle-Graves Theorem (a version is stated
as Theorem~\ref{thm:classical-Bartle-Graves} in this paper), that
there always exists a continuous selection of $\varphi$. Modern proofs
of this version of the Bartle-Graves Theorem, e.g. \cite[Corollary~17.67]{AliprantisBorder},
proceed through a straightforward application of Michael's Selection
Theorem (stated in this paper as Theorem~\ref{thm:michael-selection-theorem}).

By the Open Mapping Theorem, it can be seen that the correspondence
$\varphi$ is Lipschitz when $2^{X}$ is endowed with the Hausdorff
distance. Furthermore, the correspondence $\varphi$ also exhibits
a form of pointwise Lipschitz behaviour. Again, by the Open Mapping
Theorem, it can be seen that there exists a constant $\alpha\geq0$
so that, for any $y\in Y$ and $x\in\varphi(y)$, the correspondence
$\psi:Y\to2^{X}$, defined by $\psi(z):=\varphi(z)\cap(x+\alpha\norm{y-z}\closedball X)$
for all $z\in Y$, is non-empty-valued. All selections $\tau:Y\to X$
of $\psi$ (continuous or not) will therefore be strongly pointwise
$\alpha$-Lipschitz at $y$ (by which we mean $\norm{\tau(y)-\tau(z)}\leq\alpha\norm{y-z}$
for all $z\in Y$).

Even though the correspondence $\varphi$ \emph{always }exhibits some
form of Lipschitz behavior, an example devised by Aharoni and Lindenstrauss
(cf. \cite{LindenstraussAharoni} and \cite[Example~1.20]{LindenstraussBenyamini})
shows that it is however impossible establish the existence of Lipschitz
selections of inverse image correspondences in general. Godefroy and
Kalton gave a characterization of the continuous linear surjections
between separable Banach spaces admitting Lipschitz right inverses
as exactly the ones with continuous linear right inverses, and hence,
as exactly those with complemented kernels (cf. \cite[Corollary~3.2]{GodefroyKalton}).
However, this does not extend to non-separable Banach spaces (cf.
\cite[Section~2.2]{KaltonNonlinGeom}). We also refer the reader to
the negative result \cite[Theorem~2.4]{PrzeslawskiYost} by \preslawski\
and Yost.

Still, the Lipschitz-like behaviour of the inverse image correspondence
of a continuous linear surjection between Banach spaces (and other
related correspondences\footnote{For example, for Banach spaces $X,Y$ and a closed cone $C\subseteq X$,
consider the inverse image correspondence of a continuous additive
positively homogeneous surjection $T:C\to Y$ (cf. \cite{deJeuMesserschmidtOpenMapping}).}) raises the following question:
\begin{question*}
\label{que:more-regular-selections}Given the general Lipschitz-like
behavior of the inverse image correspondence $\varphi$, as defined
above, do there exist selections of $\varphi$ that exhibit more regularity
than the mere continuity ensured by the classical Bartle-Graves Theorem?
More generally, is there a theorem for correspondences displaying
such Lipschitz-like behavior, akin to Michael's Selection Theorem,
asserting the existence of selections which exhibit more regularity
than mere continuity?
\end{question*}
We will give one positive answer to this question in this paper. Our
main goal in this paper is to prove a general Pointwise Lipschitz
Selection Theorem (Theorem~\ref{thm:ptwise-lip-selection-theorem}).
This result gives sufficient conditions under which a correspondence
always admits a continuous selection that is pointwise Lipschitz on
a dense set of its domain. Explicitly:

\begin{thmff}{Pointwise Lipschitz Selection Theorem}{thm:ptwise-lip-selection-theorem}Let
$(M,d)$ be a metric space and $X$ a Banach space. Let $\alpha\geq0$
and let $\varphi:M\to2^{X}$ be a non-empty\textendash , closed\textendash ,
and convex-valued correspondence that admits local strongly pointwise
$\alpha$-Lipschitz selections (as defined in Definition~\ref{def:correspondence-admits-strong-ptwise-lip-selections}).
If $\varphi$ has a (bounded) continuous selection, then, for every
$\beta>\alpha$, there exists a (bounded) continuous selection of
$\varphi$ that is pointwise $\beta$-Lipschitz (as defined in Definition~\ref{def:ptwise-lip-and-strongly-ptwise-lip-def})
on a dense set of $M$.

\end{thmff}

The proof proceeds through a somewhat delicate inductive construction
which is performed in proving Lemma~\ref{lem:existence-of-successively-more-ptwise-lip-sequence}.
There we prove the existence of a uniform Cauchy sequence of continuous
selections that are pointwise Lipschitz on the points successively
larger maximal separations of $M$ (cf. Definition~\ref{def:separation}).
Each selection in this sequence is constructed as a slight adjustment
of its predecessor so as to be pointwise Lipschitz at more points.
This is done while also taking care that our adjustments do not disturb
the predecessor where it is already known to be pointwise Lipschitz.
The limit of this sequence is then shown to have the desired properties
in the proof of Theorem~\ref{thm:ptwise-lip-selection-theorem}.
We refer the reader to Section~\ref{sec:Main-result} for further
details.

The condition of admitting local strongly pointwise $\alpha$-Lipschitz
selections required in the hypothesis of the above theorem is admittedly
somewhat synthetic. The reason for working with this condition in
favor of a more natural, more easily verified condition, is purely
to abstract out the important features required in the proof of Theorem~\ref{thm:ptwise-lip-selection-theorem}.
Yet, the critical reader may well raise the question: Why should correspondences
with this property even exist? In reply, we introduce the more natural
notion of ``lower pointwise Lipschitz-ness'' of a correspondence (cf.
Definition~\ref{def:lower-ptwise-lip}). Section~\ref{sec:Lower-pointwise-Lipschitz-correspondences}
is devoted to showing that being lower pointwise Lipschitz is sufficient
for a correspondence to admit strongly pointwise Lipschitz selections.
This allows us to prove versions of Theorem~\ref{thm:ptwise-lip-selection-theorem}
in Corollaries~\ref{cor:natural-ptwise-lip-sel-thm1} and~\ref{cor:natural-ptwise-lip-sel-thm2}
which are slightly less general, yet slightly more natural.

\begin{corff}{}{cor:natural-ptwise-lip-sel-thm1}%

Let $(M,d)$ be a metric space, $X$ a Banach space and $\alpha\geq0$.
Let $\varphi:M\tocorr X$ be a closed\textendash{} and convex-valued
lower hemicontinuous correspondence that is \lowerptlip{(\alpha+\varepsilon)}
for every $\varepsilon>0$. Then, for any $\beta>\alpha$, there exists
a continuous selection of $\varphi$ that is pointwise $\beta$-Lipschitz
on a dense set of $M$.

If, additionally, there exists a bounded continuous selection of $\varphi$,
then, for any $\beta>\alpha$, there also exists a bounded continuous
selection of $\varphi$ that is pointwise $\beta$-Lipschitz on a
dense set of $M$.

\end{corff}

\begin{corff}{}{cor:natural-ptwise-lip-sel-thm2}%

Let $(M,d)$ be a \textbf{bounded} metric space, $X$ a Banach space
and $\alpha\geq0$. Let $\varphi:M\tocorr X$ be a closed\textendash{}
and convex-valued lower hemicontinuous correspondence that is \lowerptlip{(\alpha+\varepsilon)}
for every $\varepsilon>0$. Then, for any $\beta>\alpha$, there exists
a bounded continuous selection of $\varphi$ that is pointwise $\beta$-Lipschitz
on a dense set of $M$.

\end{corff}

As an illustrative application, we establish the following slightly
improved version of the classical Bartle-Graves Theorem:

\begin{thmff}{Improved Bartle-Graves Theorem}{thm:improved-Bartle-Graves}Let
$X$ and $Y$ be infinite dimensional Banach spaces and $T:X\to Y$
a continuous linear surjection. Then there exists a constant $\eta>0$
and a positively homogeneous continuous right inverse $\tau:Y\to X$
of $T$ that is pointwise $\eta$-Lipschitz on a dense \textbf{meager}
set of $Y$.

\end{thmff}

We note that Theorem~\ref{thm:ptwise-lip-selection-theorem} is,
in some sense, optimal. To elaborate, in general it is impossible
to conclude that a selection yielded by Theorem~\ref{thm:ptwise-lip-selection-theorem}
is pointwise Lipschitz on its entire domain $M$. Should this be the
case in general, a result adapted from Durand-Cartagena and Jaramillo
\cite[Corollary~2.4]{Durand-CartagenaJaramillo} in combination with
a result by Sch\"affer \cite[Theorem~3.5]{Schaffer}\footnote{Proven independently by the author and Wortel \cite[Theorem~3.6]{intrinsic-metric},
while in ignorance of Sch\"affer's work.}, will show that an inverse image correspondence will always admit
a Lipschitz selection. This however contradicts the above mentioned
example devised by Aharoni and Lindenstrauss of an inverse image correspondence
that admits no Lipschitz selection. We refer the reader to Section~\ref{sec:aharoni-lindenstrauss-example}
for further details.

\bigskip

We give a brief outline of the organization of the paper.

\medskip

In Section~\ref{sec:Preliminary-definitions,-results}, we provide
the notation and definitions used throughout this paper. Some elementary
preliminary results are also proven. Specifically, Section~\ref{subsec:Separations-in-metric-spaces}
gives some quite elementary results on so-called separations in metric
spaces, and Section~\ref{subsec:Pointwise-Lipschitz-maps} proves
some basic results on pointwise Lipschitz functions.

Section~\ref{sec:Main-result} will establish our main result, Theorem~\ref{thm:ptwise-lip-selection-theorem}.
The proof of this theorem is presented in two steps. Firstly, we give
sufficient conditions for a correspondence to have a uniform Cauchy
sequence of continuous selections, where the members of this sequence
are pointwise Lipschitz on successively finer separations (cf. Lemma~\ref{lem:existence-of-successively-more-ptwise-lip-sequence}).
The second step analyses the limit of such a Cauchy sequence of selections
and shows the limit is a selection which is pointwise Lipschitz on
a dense set of its domain (cf. Theorem~\ref{thm:ptwise-lip-selection-theorem}).

In Section~\ref{sec:Lower-pointwise-Lipschitz-correspondences} we
define the notion of lower pointwise Lipschitz-ness of a correspondence
(cf. Definition~\ref{def:lower-ptwise-lip}). This property is more
natural than that of admitting local strongly pointwise Lipschitz
selections as required in the hypothesis of Theorem~\ref{thm:ptwise-lip-selection-theorem}.
Proposition~\ref{prop:lower-(alpha+eps)-Lip-implies-admits-strong-pw-lip-selections}
shows that lower pointwise Lipschitz-ness of a correspondence is sufficient
for Theorem~\ref{thm:ptwise-lip-selection-theorem} to be applicable
to the correspondence, and results in the somewhat more natural versions
of Theorem~\ref{thm:ptwise-lip-selection-theorem} in Corollaries~\ref{cor:natural-ptwise-lip-sel-thm1}
and~\ref{cor:natural-ptwise-lip-sel-thm2}.

We give an application of our Pointwise Lipschitz Selection Theorem
in Section~\ref{sec:Applications} by establishing a slightly improved
version (Theorem~\ref{thm:improved-Bartle-Graves}) of the classical
Bartle-Graves Theorem.

Finally, in Section~\ref{sec:aharoni-lindenstrauss-example} we briefly
discuss the significance of an example devised by Aharoni and Lindenstrauss
to our results. Specifically, we argue why Theorem~\ref{thm:ptwise-lip-selection-theorem}
is, in some sense, the best possible general Lipschitz selection theorem
one can hope to prove.

\section{Preliminary definitions, results and notation\label{sec:Preliminary-definitions,-results}}

\subsection{Notation for balls in metric spaces}

For a metric space $(M,d)$ with $a\in M$ and $r>0$, we will denote
the open and closed balls with radius $r$ about $a$ respectively
by
\begin{eqnarray*}
\openball M(a,r) & := & \set{b\in M}{d(a,b)<r},\\
\closedball M(a,r) & := & \set{b\in M}{d(a,b)\leq r}.
\end{eqnarray*}

Let $X$ be a normed space. We denote the open unit ball, closed unit
ball and unit sphere respectively by $\openball X$, $\closedball X$
and $\sphere X$. To aid in readability by reducing nested parentheses,
for $x\in X$ and $r>0$, we will sometimes write $x+r\openball X$
and $x+r\closedball X$ instead of $\openball X(x,r)$ and $\closedball X(x,r)$.
We will view any subset of $X$ as a metric space with the metric
induced from the norm on $X$.

\subsection{Spaces of functions}

Let $F$ be a Hausdorff space and $X$ a normed space. By $C(F,X)$
we will denote the normed space of all bounded continuous functions
on $F$ taking values in $X$, endowed with the uniform norm $\norm{\cdot}_{\infty}$.
A standard argument shows that $C(F,X)$ is a Banach space whenever
$X$ is a Banach space. For any function $f:F\to X$ and $G\subseteq F$,
we denote the restriction of $f$ to $G$ by $\restrict fG:G\to X$.

\subsection{Correspondences}

Let $A,B$ be sets. By a \emph{correspondence} we mean a set-valued
map $\varphi:A\to2^{B}$ and we will use the notation $\varphi:A\tocorr B$.
We will say $\varphi$ is \emph{non-empty-valued} (respectively, \emph{convex-valued}
or \emph{closed-valued}) if $\varphi(a)$ is non-empty (respectively,
convex or closed) for all $a\in A$ (granted that these notions make
sense in $B$).

Let $A$ and $B$ be topological spaces and $\varphi:A\tocorr B$
be any correspondence. We will say that $\varphi$ is \emph{lower
hemicontinuous} \emph{at $a\in A$}, if, for every open set $U\subseteq B$
satisfying $\varphi(a)\cap U\neq\emptyset$, there exists some open
set $V\ni a$ satisfying $\varphi(v)\cap U\neq\emptyset$ for all
$v\in V$. We will say that $\varphi$\emph{ }is\emph{ lower hemicontinuous,
}if $\varphi$ is lower hemicontinuous at every point in $A$. By
a \emph{selection of $\varphi$ }we mean a function $f:A\to B$ satisfying
$f(a)\in\varphi(a)$ for all $a\in A.$

We quote the following two well-known classical results that we will
need in later sections.
\begin{thm}[{Michael's Selection Theorem \cite[Theorem~17.66]{AliprantisBorder}}]
\label{thm:michael-selection-theorem}Let $P$ be a paracompact space
and $X$ a Banach space. Every non-empty\textendash , closed\textendash{}
and convex-valued lower hemicontinuous correspondence $\varphi:P\tocorr X$
has a continuous selection.
\end{thm}

{}
\begin{thm}[{Stone's Theorem \cite[Corollary 1]{Stone}}]
\label{thm:stone's-theorem}Every metric space is paracompact.
\end{thm}

\subsection{Separations in metric spaces\label{subsec:Separations-in-metric-spaces}}

In this section we give some basic definitions and results concerning
separations in metric spaces. Lemmas \ref{lem:maximal-r-separations-containing-another-exist}
and \ref{lem:union-of-nested-maximal-r-separations-is-dense} are
elementary verifications whose proofs we omit.
\begin{defn}
\label{def:separation}Let $(M,d)$ be a metric space. For $r>0$,
a set $B\subseteq M$ will be called an \emph{$r$-separation} in
$M$, if, for distinct $a,b\in B$, we have $d(a,b)\geq r$. We partially
order the set of all $r$-separations in $M$ by inclusion.
\end{defn}

A straightforward application of Zorn's Lemma will establish:
\begin{lem}
\label{lem:maximal-r-separations-containing-another-exist}Let $(M,d)$
be a metric space. Let $r>0$ and let $B$ be an $r$-separation in
$M$. Then there exists a maximal $r$-separation in $M$ containing
$B$.
\end{lem}

{}
\begin{lem}
\label{lem:union-of-nested-maximal-r-separations-is-dense}Let $(M,d)$
be a metric space and let $\{r_{n}\}$ be any descending sequence
of positive real numbers that converges to zero. For each $n\in\N$,
let $B_{n}$ be a maximal $r_{n}$-separation with $B_{n-1}\subseteq B_{n}$
(where we set $B_{0}:=\emptyset$). Then $\bigcup_{n\in\N}B_{n}$
is dense in $M$.
\end{lem}

\subsection{Pointwise Lipschitz maps\label{subsec:Pointwise-Lipschitz-maps}}

In this section we introduce the notion of pointwise Lipschitz functions.
\begin{defn}
\label{def:ptwise-lip-and-strongly-ptwise-lip-def}Let $(M,d)$ and
$(M',d')$ be metric spaces, $\alpha\geq0$ and $f:M\to M'$ any map.
\begin{enumerate}
\item We will say $f$ is \emph{pointwise $\alpha$-Lipschitz at $b\in M$,}
if
\[
\limsup_{r\to0^{+}}\parenth{r^{-1}\sup\set{d'(f(b),f(a))}{a\in\closedball M(b,r)}}\leq\alpha.
\]
For a set $S\subseteq M$, we will say $f$ is \emph{pointwise $\alpha$-Lipschitz
on $S$} if $f$ is pointwise $\alpha$-Lipschitz at every point of
$S$.\medskip
\item We will say $f$ is \emph{strongly pointwise $\alpha$-Lipschitz}\footnote{The term \emph{calmness }also occurs in the literature \cite[Section~1.3]{DontchevRockafellarImplicitFunctions}.}\emph{
at $b\in M$,} if, for all $a\in M$,
\[
d'(f(b),f(a))\leq\alpha d(b,a).
\]
\end{enumerate}
\end{defn}

The following somewhat technical lemmas will be needed in later sections.
In summary, Lemma~\ref{lem:pt-wise-lip-the-same-for-closed-or-open-balls}
shows that we may replace closed balls with open balls in the definition
of pointwise Lipschitz-ness, and Lemma~\ref{lem:homogeneous-exstension-preserves-pw-lip}
shows that pointwise Lipschitz-ness is preserved by homogeneous extensions
of bounded functions.
\begin{lem}
\label{lem:pt-wise-lip-the-same-for-closed-or-open-balls}Let $(M,d)$
and $(M',d')$ be metric spaces and $\alpha\geq0$. For some $b\in M$,
a function $f:M\to M'$ is pointwise $\alpha$-Lipschitz at $b$ if
and only if
\[
\limsup_{r\to0^{+}}\parenth{r^{-1}\sup\set{d(f(b),f(a))}{a\in\openball M(b,r)}}\leq\alpha.
\]
\end{lem}

\begin{proof}
Let $f$ be pointwise $\alpha$-Lipschitz at $b$. Since $\openball M(b,r)\subseteq\closedball M(b,r)$
for all $r>0$, we have
\begin{eqnarray*}
 &  &\mkern-72mu \limsup_{r\to0^{+}}\parenth{r^{-1}\sup\set{d(f(b),f(a))}{a\in\openball M(b,r)}}\\
 & \leq & \limsup_{r\to0^{+}}\parenth{r^{-1}\sup\set{d(f(b),f(a))}{a\in\closedball M(b,r)}}\\
 & \leq & \alpha.
\end{eqnarray*}

Conversely, let $\limsup_{r\to0^{+}}\parenth{r^{-1}\sup\set{d(f(b),f(a))}{a\in\openball M(b,r)}}\leq\alpha.$
Let $\varepsilon>0$ be arbitrary. Then there exists some $s>0$ such
that, for all $r\in(0,s)$, we have
\[
r^{-1}\sup\set{d(f(b),f(a))}{a\in\openball M(b,r)}<\alpha+2^{-1}\varepsilon.
\]
Let $r\in(0,s)$ be arbitrary. Then, for any $\kappa>0$ satisfying
\[
\kappa<\min\{2^{-1}r\varepsilon(\alpha+2^{-1}\varepsilon)^{-1},s-r\},
\]
we have $0<r+\kappa<s$, and hence
\begin{eqnarray*}
 &  &\mkern-72mu r^{-1}\sup\set{d(f(b),f(a))}{a\in\closedball M(b,r)}\\
 & \leq & r^{-1}\sup\set{d(f(b),f(a))}{a\in\openball M(b,r+\kappa)}\\
 & = & r^{-1}(r+\kappa)(r+\kappa)^{-1}\sup\set{d(f(b),f(a))}{a\in\openball M(b,r+\kappa)}\\
 & < & r^{-1}(r+\kappa)(\alpha+2^{-1}\varepsilon)\\
 & < & (\alpha+2^{-1}\varepsilon)+\kappa r^{-1}(\alpha+2^{-1}\varepsilon)\\
 & < & \alpha+\frac{1}{2}\varepsilon+\frac{1}{2}\varepsilon\\
 & = & \alpha+\varepsilon.
\end{eqnarray*}
Since $r\in(0,s)$ was chosen arbitrarily, we obtain
\[
\sup_{r\in(0,s)}\parenth{r^{-1}\sup\set{d(f(b),f(a))}{a\in\closedball M(b,r)}}\leq\alpha+\varepsilon.
\]
Since $\varepsilon>0$ was chosen arbitrarily, we obtain
\[
\limsup_{r\to0^{+}}\parenth{r^{-1}\sup\set{d(f(b),f(a))}{a\in\closedball M(b,r)}}\leq\alpha.\qedhere
\]
\end{proof}
\begin{lem}
\label{lem:homogeneous-exstension-preserves-pw-lip}\global\long\def\underrho{\underline{\rho}}
Let $Y$ and $X$ be normed spaces. Let $y\in\sphere Y$, $\beta>0$,
and let $\underrho\in C(\sphere Y,X)$ be pointwise $\beta$-Lipschitz
at $y$. Then the positively homogeneous extension $\rho:Y\to X$
of $\underrho$, defined by
\[
\rho(z):=\begin{cases}
0 & \text{if }z=0\\
\norm z\underline{\rho}\parenth{\frac{z}{\norm z}} & \text{if }z\neq0,
\end{cases}\quad(z\in Y),
\]
is continuous and is pointwise $(2\beta+\norm{\underline{\rho}}_{\infty})$-Lipschitz
on the set $\{\lambda y\in Y | \lambda>0\}.$
\end{lem}

\begin{proof}
That $\rho$ is continuous is a straightforward exercise using reverse
triangle inequality and the boundedness of $\underline{\rho}$.

Let $\varepsilon>0$ be arbitrary. Since $\underrho$ is pointwise
$\beta$-Lipschitz at $y$, there exists some $R\in(0,1)$, such that,
for all $r\in(0,R)$,
\[
r^{-1}\sup\set{\norm{\underline{\rho}\parenth y-\underline{\rho}\parenth x}}{x\in\closedball Y(y,r)\cap\sphere Y}<\beta+2^{-1}\varepsilon.
\]

Let $z\in\set{\lambda y\in Y}{\lambda>0}$ and $s\in(0,2^{-1}\norm zR)$
be arbitrary. For any $x\in\closedball Y(z,s)$, we note that $x\neq0$,
since $s<\norm z$. Furthermore,
\begin{eqnarray*}
\norm{\frac{z}{\norm z}-\frac{x}{\norm x}} & \leq & \norm{\frac{z}{\norm z}-\frac{x}{\norm z}}+\norm{\frac{x}{\norm z}-\frac{x}{\norm x}}\\
 & = & \frac{1}{\norm z}\norm{z-x}+\abs{\frac{1}{\norm z}-\frac{1}{\norm x}}\norm x\\
 & \leq & \frac{s}{\norm z}+\abs{\frac{\norm x-\norm z}{\norm z\norm x}}\norm x\\
 & \leq & \frac{s}{\norm z}+\frac{1}{\norm z}\norm{z-x}\\
 & \leq & \frac{2s}{\norm z}\\
 & < & R.
\end{eqnarray*}
Therefore, for any $x\in\closedball Y(z,s)$, we have
\[
\norm{y-\frac{x}{\norm x}}=\norm{\frac{z}{\norm z}-\frac{x}{\norm x}}\leq\frac{2s}{\norm z}<R,
\]
and hence,
\begin{eqnarray*}
s^{-1}\norm{\rho(z)-\rho(x)} & = & s^{-1}\norm{\norm z\underrho\parenth{\frac{z}{\norm z}}-\norm x\underrho\parenth{\frac{x}{\norm x}}}\\
 & \leq & s^{-1}\norm{\norm z\underrho\parenth{\frac{z}{\norm z}}-\norm z\underrho\parenth{\frac{x}{\norm x}}}\\
 &  & \quad\quad+s^{-1}\norm{\norm z\underrho\parenth{\frac{x}{\norm x}}-\norm x\underrho\parenth{\frac{x}{\norm x}}}\\
 & = & s^{-1}\norm z\norm{\underrho\parenth{\frac{z}{\norm z}}-\underrho\parenth{\frac{x}{\norm x}}}\\
 &  & \quad\quad+s^{-1}\abs{\vphantom{\sum}\norm z-\norm x}\norm{\underrho\parenth{\frac{x}{\norm x}}}\\
 & \leq & s^{-1}\norm z\norm{\underrho\parenth{\frac{z}{\norm z}}-\underrho\parenth{\frac{x}{\norm x}}}+s^{-1}\norm{z-x}\norm{\underrho}_{\infty}\\
 & \leq & 2\parenth{\frac{2s}{\norm z}}^{-1}\norm{\underrho\parenth y-\underrho\parenth{\frac{x}{\norm x}}}+s^{-1}s\norm{\underrho}_{\infty}\\
 & < & 2(\beta+2^{-1}\varepsilon)+\norm{\underrho}_{\infty}\\
 & = & 2\beta+\norm{\underrho}_{\infty}+\varepsilon.
\end{eqnarray*}
Since $s\in(0,2^{-1}\norm zR)$ was chosen arbitrarily, we obtain
\[
\sup_{s\in(0,2^{-1}\norm zR)}\parenth{s^{-1}\set{\norm{\rho(z)-\rho(x)}}{x\in\closedball Y(z,s)}}\leq2\beta+\norm{\underrho}_{\infty}+\varepsilon.
\]
Since $\varepsilon>0$ was chosen arbitrarily, we have
\[
\limsup_{r\to0^{+}}\parenth{r^{-1}\set{\norm{\rho(z)-\rho(x)}}{x\in\closedball Y(z,r)}}\leq2\beta+\norm{\underrho}_{\infty}.
\]
Finally, since $z$ was chosen arbitrarily from $\set{\lambda y\in Y}{\lambda>0}$,
we conclude that $\rho$ is pointwise $\parenth{2\beta+\norm{\underrho}_{\infty}}$-Lipschitz
on $\set{\lambda y\in Y}{\lambda>0}$.
\end{proof}

\section{Main result: A Pointwise Lipschitz Selection Theorem\label{sec:Main-result}}

In this section we will prove our Pointwise Lipschitz Selection Theorem
(Theorem~\ref{thm:ptwise-lip-selection-theorem}).

For the sake of brevity and clarity of the proof, the results in this
section is stated under the somewhat synthetic assumption of a correspondence
admitting local strongly pointwise Lipschitz selections. Section~\ref{sec:Lower-pointwise-Lipschitz-correspondences}
introduces a more natural property which we call lower pointwise Lipschitz-ness
which allows for the statement of more natural versions of Theorem~\ref{thm:ptwise-lip-selection-theorem}.
\begin{defn}
\label{def:correspondence-admits-strong-ptwise-lip-selections}Let
$\alpha\geq0$ and $(M,d)$ and $(M',d')$ be metric spaces and let
$b\in M$. Let $\varphi:M\tocorr M'$ be a correspondence. We will
say that $\varphi$ \emph{admits }\textbf{\emph{local}}\emph{ strongly
pointwise $\alpha$-Lipschitz selections at $b$} if, for every $y\in\varphi(b)$,
there exists some open neighborhood $U\subseteq M$ of $b$ and a
continuous selection $f:M\to M'$ of $\varphi$ satisfying $f(b)=y$
with the restriction $\restrict fU:U\to M'$ strongly pointwise $\alpha$-Lipschitz
at $b$ (as defined in Definition~\ref{def:ptwise-lip-and-strongly-ptwise-lip-def}).

We will say $\varphi$ \emph{admits }\textbf{\emph{local}}\emph{ strongly
pointwise $\alpha$-Lipschitz selections, }if it admits local strongly
pointwise $\alpha$-Lipschitz selections at every point of $M$. If
we may choose the neighborhood $U$ as the whole space $M$, we will
omit the `local' modifier, by saying $\varphi$ \emph{admits strongly
pointwise $\alpha$-Lipschitz selections (at $b$).}
\end{defn}

With this definition in hand, we can turn toward establishing our
Pointwise Lipschitz Selection Theorem (Theorem~\ref{thm:ptwise-lip-selection-theorem}).
The proof is somewhat delicate and is split into two parts. We briefly
describe the argument employed:

The first and most technical part is given in the proof of Lemma~\ref{lem:existence-of-successively-more-ptwise-lip-sequence}.
Given a non-empty\textendash{} and convex-valued correspondence that
admits local strongly pointwise Lipschitz selections, we start with
any continuous selection $f_{0}$ of this correspondence. We inductively
construct a sequence of selections $\{f_{n}\}$ of the correspondence
in such a way that, for each $n\in\N$, the selection $f_{n}$ is
pointwise Lipschitz at more points than its predecessor in the sequence
$f_{n-1}$. This is achieved by making subtle adjustments to $f_{n-1}$.
It is necessary to use a delicate hand in the construction of $f_{n}$
from $f_{n-1}$ to ensure that one does not disturb $f_{n-1}$ at
the points where it is already pointwise Lipschitz. We do this by
carefully adjusting a selection in the sequence only at points that
form part of a sequence of successively finer maximal separations
(denoted by $\{B_{n}\}$ in Lemma~\ref{lem:existence-of-successively-more-ptwise-lip-sequence}).
This process yields precise control over the distance from the points
where $f_{n-1}$ is already pointwise Lipschitz and points where it
is safe to adjust $f_{n-1}$. We exploit this control together with
a standard partition of unity argument and the assumption that the
correspondence admits strongly pointwise Lipschitz selections to then
carefully adjust $f_{n-1}$ to form its successor $f_{n}$.

The second part is given in the proof of Theorem~\ref{thm:ptwise-lip-selection-theorem}.
Using a sequence of selections of the correspondence $\{f_{n}\}$,
as obtained from Lemma~\ref{lem:existence-of-successively-more-ptwise-lip-sequence},
it is easily seen that this sequence is uniform Cauchy and hence converges
to a continuous selection of the correspondence. The bulk of the proof
of Theorem~\ref{thm:ptwise-lip-selection-theorem} is a verification
of the properties of the limit of this sequence, in particular that
it is pointwise Lipschitz on a dense set of its domain.
\begin{lem}
\label{lem:existence-of-successively-more-ptwise-lip-sequence}Let
$(M,d)$ be a metric space and $X$ a normed space. Let $\alpha\geq0$
and $\varphi:M\tocorr X$ be a non-empty\textendash{} and convex-valued
correspondence that admits local strongly pointwise $\alpha$-Lipschitz
selections (as defined in Definition~\ref{def:correspondence-admits-strong-ptwise-lip-selections}).
Then, for every $\varepsilon>0$ and any continuous selection $f_{0}$
of $\varphi$, there exists a sequence of continuous functions $\{f_{n}:M\to X\}$
and a sequence of subsets $\{B_{n}\}$ of $M$ such that, for every
$n\in\N$:
\begin{enumerate}
\item The set $B_{n}$ is a maximal $2^{-(n-1)}$-separation in $M$ with
$B_{n-1}\subseteq B_{n}$ (where we take $B_{0}:=\emptyset$).\medskip
\item The function $f_{n}$ is a continuous selection of $\varphi$. If
$f_{0}$ is bounded, then so is $f_{n}.$\medskip
\item We have $\sup_{a\in M}\norm{f_{n}(a)-f_{n-1}(a)}\leq2^{-n}\varepsilon$.
\medskip
\item The function $f_{n}$ is pointwise $\alpha$-Lipschitz at every point
of $B_{n}$.\medskip
\item For every $b\in B_{n}\backslash B_{n-1}$ there exists a number $\delta_{b}^{(n)}>0$
so that, for every $a\in\openball M(b,\delta_{b}^{(n)})$, we have
$\norm{f_{n}(b)-f_{n}(a)}\leq\alpha d(b,a)$.\medskip
\item For any $k\in\{1,\ldots,n-1\}$ and $b\in B_{k}$, the function $f_{n}$
coincides with all the functions $f_{n-1},\ldots,f_{k}$ on $\openball M(b,2^{-n})$,
that is,
\[
\restrict{f_{n}}{\openball M(b,2^{-n})}=\restrict{f_{n-1}}{\openball M(b,2^{-n})}=\ldots=\restrict{f_{k}}{\openball M(b,2^{-n})}.
\]

\end{enumerate}
\end{lem}

\begin{proof}
Let $\varepsilon>0$ be arbitrary. Let $B_{0}:=\emptyset$ and let
$f_{0}:M\to X$ be a continuous selection of $\varphi$.

We proceed inductively. Let $n\in\N$ be arbitrary, and assume that
the functions $f_{1},\ldots f_{n-1}$ and the sets $B_{1},\ldots,B_{n-1}$ have been
defined to satisfy (1)-(6) in the statement of the result. We will
construct $f_{n}:M\to X$ and $B_{n}$.

Firstly, by Lemma~\ref{lem:maximal-r-separations-containing-another-exist},
there exists a maximal $2^{-(n-1)}$-separation in $M$, denoted by
$B_{n}$, satisfying $B_{n-1}\subseteq B_{n}$.

For every $b\in B_{n}\backslash B_{n-1}$, by our assumption of $\varphi$
admitting local strongly pointwise $\alpha$-Lipschitz selections,
there exists some $r_{b}>0$ and a continuous selection $g_{b}:M\to X$
of $\varphi$ satisfying $g_{b}(b)=f_{n-1}(b)$ with the restriction
$\restrict{g_{b}}{\openball M(b,r_{b})}$ strongly pointwise $\alpha$-Lipschitz
at $b$. In particular, $g_{b}$ is pointwise $\alpha$-Lipschitz
at $b$. The map $M\ni a\mapsto\norm{f_{n-1}(a)-g_{b}(a)}$ is continuous,
hence there exists some $\delta_{b}^{(n)}>0$ with $\delta_{b}^{(n)}<\min\{2^{-(n+1)},r_{b}\}$
such that, if $a\in\openball M(b,2\delta_{b}^{(n)})$, then
\[
\abs{\vphantom{\sum}\norm{f_{n-1}(a)-g_{b}(a)}-\norm{f_{n-1}(b)-g_{b}(b)}}<2^{-n}\varepsilon.
\]
But, since $f_{n-1}(b)=g_{b}(b)$, we have $\norm{f_{n-1}(a)-g_{b}(a)}<2^{-n}\varepsilon$
for every $a\in\openball M(b,2\delta_{b}^{(n)})$.

We define the collections
\[
\cal U:=\set{\openball M(b,2\delta_{b}^{(n)})}{b\in B_{n}\backslash B_{n-1}}\ \mbox{and}\ \cal C:=\set{\closedball M(b,\delta_{b}^{(n)})}{b\in B_{n}\backslash B_{n-1}}.
\]
Since $B_{n}$ is a $2^{-(n-1)}$-separation in $M$ and, for every
$b\in B_{n}\backslash B_{n-1}$, we have $\delta_{b}^{(n)}<2^{-(n+1)}$,
the elements of $\cal U$ are pairwise disjoint. Similarly, the elements
of $\cal C$ are pairwise disjoint. Furthermore, it can be seen that
$\cup\cal C$ is closed (see \cite[III.9.2]{Dugundji}).

We define $\cal V:=\cal U\cup\{M\backslash\cup\cal C\}$, which is
an open cover of $M$. Since $M$ is paracompact (cf. Theorem~\ref{thm:stone's-theorem}),
there exists a locally finite partition of unity $\set{\rho_{V}}{V\in\cal V}$
subordinate to $\cal V$ \cite[Theorem~4.2, p.170]{Dugundji}. For
$V\in\cal V$, if $V=M\backslash\cup\cal C$, we define $h_{V}:=f_{n-1}$.
If $V=\openball M(b,2\delta_{b}^{(n)})$ for some $b\in B_{n}\backslash B_{n-1}$,
we define $h_{V}:=g_{b}$. Finally, we define $f_{n}:M\to X$ as
\[
f_{n}(a):=\sum_{V\in\cal V}\rho_{V}(a)h_{V}(a)\quad(a\in M).
\]
Since $\varphi$ is convex-valued, and, for every $V\in\cal V$, the
function $h_{V}$ is a continuous selection of $\varphi$, we have
that $f_{n}$ is a continuous selection of $\varphi$.

We claim that $\sup_{a\in M}\norm{f_{n-1}(a)-f_{n}(a)}\leq2^{-n}\varepsilon$.
Let $a\in M$ be arbitrary. We distinguish two cases: Firstly, if
$a\notin\cup\cal U$, then $a\in M\backslash\cup\cal C$ and therefore
$\norm{f_{n-1}(a)-f_{n}(a)}=0$. Secondly, if $a\in\cup\cal U$, then,
since the elements of $\cal U$ are disjoint, there exists a unique
$b\in B_{n}\backslash B_{n-1}$ so that $a\in\openball M(b,2\delta_{b}^{(n)})=:V$
and $h_{V}=g_{b}$. Then, with $W:=M\backslash\cup\cal C$, we have
$h_{W}=f_{n-1}$. By definition of $\delta_{b}^{(n)}$, we see that
\begin{eqnarray*}
\norm{f_{n-1}(a)-f_{n}(a)} & = & \norm{f_{n-1}(a)-\sum_{U\in\{V,W\}}\rho_{U}(a)h_{U}(a)}\\
 & \leq & \abs{\rho_{V}(a)}\norm{f_{n-1}(a)-h_{V}(a)}\\
 &  & \quad\quad+\abs{\rho_{W}(a)}\norm{f_{n-1}(a)-h_{W}(a)}\\
 & \leq & \norm{f_{n-1}(a)-g_{b}(a)}\\
 & < & 2^{-n}\varepsilon,
\end{eqnarray*}
establishing our claim that $\sup_{a\in M}\norm{f_{n-1}(a)-f_{n}(a)}\leq2^{-n}\varepsilon$.
If $f_{0}$ is bounded, then $f_{n-1}$ is bounded by assumption,
and hence it is clear that $f_{n}$ is also bounded.

We notice, by construction, for every $b\in B_{n}\backslash B_{n-1}$
we have $\openball M(b,\delta_{b}^{(n)})\cap(M\backslash\cup\cal C)=\emptyset$.
Therefore $\restrict{f_{n}}{\openball M(b,\delta_{b}^{(n)})}=\restrict{g_{b}}{\openball M(b,\delta_{b}^{(n)})}.$
Because $\delta_{b}^{(n)}<r_{b}$, the restriction $\restrict{g_{b}}{\openball M(b,\delta_{b}^{(n)})}$
is strongly pointwise $\alpha$-Lipschitz at $b$, and hence, the
map $f_{n}$ is pointwise $\alpha$-Lipschitz at $b$.

Again by construction, for any $k\in\{1,\ldots,n-1\}$ and $b\in B_{k}\subseteq B_{n-1}$,
we have $\openball M(b,2^{-n})\cap(\cup\cal U)=\emptyset$ so that
$\restrict{f_{n}}{\openball M(b,2^{-n})}=\restrict{f_{n-1}}{\openball M(b,2^{-n})}.$
Since $f_{n-1}$ was assumed to be pointwise $\alpha$-Lipschitz at
$b$, so is $f_{n}$. Furthermore, if $k<n-1$, by our initial assumption,
we have
\[
\restrict{f_{n-1}}{\openball M(b,2^{-(n-1)})}=\restrict{f_{n-2}}{\openball M(b,2^{-(n-1)})}=\ldots=\restrict{f_{k}}{\openball M(b,2^{-(n-1)})},
\]
and since $\openball M(b,2^{-n})\subseteq\openball M(b,2^{-(n-1)})$,
we conclude
\[
\restrict{f_{n}}{\openball M(b,2^{-n})}=\restrict{f_{n-1}}{\openball M(b,2^{-n})}=\ldots=\restrict{f_{k}}{\openball M(b,2^{-n})}.\qedhere
\]
\end{proof}
\begin{rem}
\label{rem:sequence-eventually-constant}With $\{f_{n}\}$ and $\{B_{n}\}$
as constructed in the previous lemma, we note, for every $b\in\bigcup_{n\in\N}B_{n}$,
the sequence $\{f_{n}(b)\}\subseteq X$ is eventually constant. Specifically,
if for some $n\in\N$, we have $b\in B_{n}$, then $f_{m}(b)=f_{n}(b)$
for all $m\geq n$. We will use this fact in the proof of Theorem~\ref{thm:ptwise-lip-selection-theorem}.
\end{rem}

\begin{thm}[Pointwise Lipschitz Selection Theorem]
\label{thm:ptwise-lip-selection-theorem}Let $(M,d)$ be a metric
space and $X$ a Banach space. Let $\alpha\geq0$ and let $\varphi:M\tocorr X$
be a non-empty\textendash , closed\textendash , and convex-valued
correspondence that admits local strongly pointwise $\alpha$-Lipschitz
selections (as defined in Definition~\ref{def:correspondence-admits-strong-ptwise-lip-selections}).
If $\varphi$ has a (bounded) continuous selection, then, for every
$\beta>\alpha$, there exists a (bounded) continuous selection of
$\varphi$ that is pointwise $\beta$-Lipschitz (as defined in Definition~\ref{def:ptwise-lip-and-strongly-ptwise-lip-def})
on a dense set of $M$.
\end{thm}

\begin{proof}
Let $f_{0}:M\to X$ be a continuous selection of $\varphi$ and let
$\beta>\alpha$ be arbitrary. Define $\varepsilon:=3^{-1}(\beta-\alpha)>0$,
and with this $\varepsilon$ and $f_{0}$, let $\{f_{n}\}$ and $\{B_{n}\}$
be as obtained from Lemma~\ref{lem:existence-of-successively-more-ptwise-lip-sequence}.
Since the family of sequences $\{f_{n}(a)\}\subseteq X$ are uniformly
Cauchy with respect to $a\in M$, a standard exercise shows that the
pointwise limit $f:M\to X$ defined by $f(a):=\lim_{n\to\infty}f_{n}(a)$
for all $a\in M$ is continuous. Since $\varphi$ is closed-valued,
and each $f_{n}$ is a continuous selection of $\varphi,$ the limit
$f$ is also a continuous selection of $\varphi$. If $f_{0}$ is
bounded, the sequence $\{f_{n}\}$ is Cauchy in the Banach space $C(M,X)$,
and hence the limit $f$ is also bounded.

We let $B:=\bigcup_{n\in\N}B_{n}$, and by Lemma~\ref{lem:union-of-nested-maximal-r-separations-is-dense},
the set $B$ is dense in $M$. We claim that $f$ is pointwise $\beta$-Lipschitz
on $B$.

Let $b\in B$ be arbitrary. Let $N\in\N$ be the least number such
that $b\in B_{N}$. With $\delta_{b}^{(N)}>0$ as obtained from Lemma~\ref{lem:existence-of-successively-more-ptwise-lip-sequence},
let $K\in\N$ be the least number satisfying $K\geq N$ and $2^{-K}<\delta_{b}^{(N)}$.
Let $r\in(0,2^{-K})$ be arbitrary, and let $n\in\N$ be such that
$n\geq K$ and $r\in[2^{-(n+1)},2^{-n})$. Let $y\in\openball M(b,r)$
be arbitrary. If $y=b$, then $r^{-1}\norm{f(b)-f(y)}=0$. On the
other hand, if $y\neq b$, since $\openball M(b,r)$ is open and $f$
is continuous, there exists some $\delta>0$ such that both $\openball M(y,\delta)\subseteq\openball M(b,r)$,
and $\norm{f(y)-f(x)}<r\varepsilon$ for all $x\in\openball M(y,\delta)$.
By the density of $B$ in $M$, there exists some $c\in\openball M(y,\delta)\cap B$.
Let $m\geq n$ be such that $c\in B_{m}$.

Now, by construction of the sequence $\{f_{n}\}$ in Lemma~\ref{lem:existence-of-successively-more-ptwise-lip-sequence}
(cf. Remark~\ref{rem:sequence-eventually-constant}), we have $f(b)=f_{n}(b)$
and $f(c)=f_{m}(c)$, and
\[
\restrict{f_{n}}{\openball M(b,2^{-n})}=\restrict{f_{n-1}}{\openball M(b,2^{-n})}=\ldots=\restrict{f_{N}}{\openball M(b,2^{-n})}.
\]
Furthermore, since
\[
c\in\openball M(y,\delta)\subseteq\openball M(b,r)\subseteq\openball M(b,2^{-n})\subseteq\openball M(b,\delta_{b}^{(N)}),
\]
again by Lemma~\ref{lem:existence-of-successively-more-ptwise-lip-sequence},
we have $\norm{f_{N}(b)-f_{N}(c)}\leq\alpha d(b,c)<\alpha r$. Finally,
keeping in mind that $\sup_{a\in M}\norm{f_{j}(a)-f_{j-1}(a)}<2^{-j}\varepsilon$
for all $j\in\N$, and that $r\geq2^{-(n+1)}$, we obtain
\begin{eqnarray*}
& &\mkern-72mu r^{-1}\norm{f(b)-f(y)} \\
 & \leq & r^{-1}\norm{f(b)-f(c)}+r^{-1}\norm{f(c)-f(y)}\\
 & = & r^{-1}\norm{f(b)-f(c)}+r^{-1}r\varepsilon\\
 & = & r^{-1}\norm{f_{n}(b)-f_{m}(c)}+\varepsilon\\
 & \leq & r^{-1}\norm{f_{n}(b)-f_{n}(c)}+\parenth{r^{-1}\sum_{j=n+1}^{m}\norm{f_{j-1}(c)-f_{j}(c)}}+\varepsilon\\
 & \leq & r^{-1}\norm{f_{n}(b)-f_{n}(c)}+\parenth{r^{-1}\sum_{j=n+1}^{m}2^{-j}\varepsilon}+\varepsilon\\
 & < & r^{-1}\norm{f_{n}(b)-f_{n}(c)}+2^{n+1}2^{-n}\varepsilon+\varepsilon\\
 & = & r^{-1}\norm{f_{N}(b)-f_{N}(c)}+3\varepsilon\\
 & \leq & r^{-1}\alpha d(b,c)+3\varepsilon\\
 & < & r^{-1}r\alpha+3\varepsilon\\
 & = & \alpha+\beta-\alpha.\\
 & = & \beta.
\end{eqnarray*}

Since $y\in\openball M(b,r)$ was chosen arbitrarily, we have
\[
r^{-1}\sup\set{\norm{f(b)-f(y)}}{y\in\openball M(b,r)}\leq\beta.
\]
But $r\in(0,2^{-K})$ was also chosen arbitrarily, and therefore
\[
\limsup_{r\to0^{+}}\parenth{r^{-1}\sup\set{\norm{f(b)-f(y)}}{y\in\openball M(b,r)}}\leq\beta.
\]
By Lemma~\ref{lem:pt-wise-lip-the-same-for-closed-or-open-balls},
the function $f$ is pointwise $\beta$-Lipschitz at $b\in B$. Finally,
since $b\in B$ was chosen arbitrarily, we conclude that $f$ is pointwise
$\beta$-Lipschitz on $B$ which is dense in $M$.
\end{proof}
\begin{rem}
\label{rem:cantor-ternary}We point out that a function that is pointwise
$\alpha$-Lipschitz for some $\alpha>0$ on a dense set of its domain
is not necessarily pointwise $\beta$-Lipschitz for some $\beta>0$
on the whole of its domain. The Cantor function \cite[Exercise 1.6.48]{TaoMeasureTheory},
which maps $[0,1]$ to $[0,1]$, is pointwise $0$-Lipschitz on the
complement of the Cantor set (which is dense in $[0,1]$), while it
is not pointwise $\alpha$-Lipschitz for any $\alpha>0$ on the whole
interval. An easy way to see this is to apply Theorem~\ref{thm:durant-cartagena-jaramillo},
noting that the Cantor function is not Lipschitz, while $[0,1]$ is
a length space (cf. \cite[Definition~2.1.6]{BuragoBuragoIvanov}).
\end{rem}

\section{Lower pointwise Lipschitz correspondences \label{sec:Lower-pointwise-Lipschitz-correspondences}}

The condition of a correspondence admitting local strongly pointwise
Lipschitz selections in the hypothesis Theorem~\ref{thm:ptwise-lip-selection-theorem}
is admittedly somewhat synthetic. In this section we will show that
a more natural condition, which we call ``lower pointwise Lipschitz-ness''
of a correspondence, is a sufficient condition for a correspondence
to admit local strongly pointwise Lipschitz selections.
\begin{defn}
\label{def:lower-ptwise-lip}Let $(M,d)$ be a metric space, $X$
a normed space, and $\alpha\geq0$. A correspondence $\varphi:M\tocorr X$
will be said to be \emph{\lowerptlip{\alpha} at $b\in M$}, if, for
every $y\in\varphi(b)$ and $a\in M$, the set
\[
\varphi(a)\cap\parenth{y+\alpha d(b,a)\closedball X}
\]
is non-empty. We will say that $\varphi$ is \emph{\lowerptlip{\alpha}}
if it is \lowerptlip{\alpha} at every point of $M$.
\end{defn}

The following few results are fairly straightforward in nature, if
somewhat technical. Our aim is to prove Proposition~\ref{prop:lower-(alpha+eps)-Lip-implies-admits-strong-pw-lip-selections}
which gives sufficient conditions for a correspondence to admit strongly
pointwise Lipschitz selections. The bulk of the work is done in Proposition~\ref{prop:lower-pt-lip-implies-lower-hemicontinuous-of-psi}
which establishes the lower hemicontinuity of a certain correspondence
derived from one that is assumed to be lower pointwise Lipschitz.
A straightforward application of Michael's Selection Theorem will
then establish Proposition~\ref{prop:lower-(alpha+eps)-Lip-implies-admits-strong-pw-lip-selections}.
\begin{lem}
\label{lem:convex-intersection-withball-nonempty}Let $X$ be a normed
space. Let $\alpha\geq0$, $x\in X$ and $G\subseteq X$ be a convex
set such that, for every $\varepsilon>0$, the set $G\cap(x+(\alpha+\varepsilon)\closedball X)$
is non-empty. If, for an open set $U\subseteq X$ and some $\varepsilon_{0}>0$,
the set $G\cap(x+(\alpha+\varepsilon_{0})\closedball X)\cap U$ is
non-empty, then $G\cap(x+(\alpha+\varepsilon_{0})\openball X)\cap U$
is also non-empty.
\end{lem}

\begin{proof}
Let $U\subseteq X$ be open and $\varepsilon_{0}>0$ such that $G\cap(x+(\alpha+\varepsilon_{0})\closedball X)\cap U\neq\emptyset$.
Let $y\in G\cap(x+(\alpha+\varepsilon_{0})\closedball X)\cap U$.
If $y\in x+(\alpha+\varepsilon_{0})\openball X$, then we are done.
We therefore assume that $y\in x+(\alpha+\varepsilon_{0})\sphere X$.
Let $z\in G\cap(x+(\alpha+2^{-1}\varepsilon_{0})\closedball X)\neq\emptyset$.
Then, for every $t\in(0,1]$,
\begin{eqnarray*}
\norm{tz+(1-t)y-x} & = & \norm{tz+(1-t)y-tx-(1-t)x}\\
 & \leq & t\norm{z-x}+(1-t)\norm{y-x}\\
 & \leq & t(\alpha+2^{-1}\varepsilon_{0})+(1-t)(\alpha+\varepsilon_{0})\\
 & < & t(\alpha+\varepsilon_{0})+(1-t)(\alpha+\varepsilon_{0})\\
 & = & (\alpha+\varepsilon_{0}).
\end{eqnarray*}
In other words, $tz+(1-t)y\in x+(\alpha+\varepsilon_{0})\openball X$
for all $t\in(0,1]$. Since $[0,1]\ni t\mapsto tz+(1-t)y$ is continuous,
there exists some $t_{0}\in(0,1]$ such that $t_{0}z+(1-t_{0})y\in U$.
Since $G$ is convex, $t_{0}z+(1-t_{0})y\in G$. We conclude $t_{0}z+(1-t_{0})y\in G\cap(x+(\alpha+\varepsilon_{0})\openball X)\cap U$.
\end{proof}
\begin{prop}
\label{prop:lower-pt-lip-implies-lower-hemicontinuous-of-psi}Let
$(M,d)$ be a metric space, $X$ a normed space and $\alpha\geq0$.
Let $a\in M$ and let $\varphi:M\tocorr X$ be a convex-valued lower
hemicontinuous correspondence that is \lowerptlip{(\alpha+\varepsilon)}
at $a\in M$ for every $\varepsilon>0$. Then, for every $y\in\varphi(a)$
and $\varepsilon>0$, the correspondence $\psi:M\tocorr X$, defined
by
\[
\psi(b):=\varphi(b)\cap\parenth{y+(\alpha+\varepsilon)d(a,b)\closedball X}\quad(b\in M),
\]
is lower hemicontinuous. Moreover, $\psi$ is convex\textendash{}
and non-empty-valued.
\end{prop}

\begin{proof}
Let $y\in\varphi(a)$ and $\varepsilon>0$ be arbitrary and let $\psi:M\tocorr X$
be as defined in the statement of the result. That $\psi$ is convex-valued
is immediate. That $\psi$ is non-empty-valued, follows from $\varphi$
being \lowerptlip{(\alpha+\varepsilon)} at $a\in M$\@.

We first show that $\psi$ is lower hemicontinuous at $a$. Let $U\subseteq X$
be an open set satisfying $\psi(a)\cap U\neq\emptyset.$ Since $\psi(a)=\{y\},$
we have $y\in U$. Let $r>0$ be such that $y+r\openball X\subseteq U$.
Since $\varphi$ is lower hemicontinuous, there exists some neighborhood
$V\subseteq M$ of $a$ so that $b\in V$ implies that $\varphi(b)\cap(y+r\openball X)\neq\emptyset$.
Let $0<s<(\alpha+\varepsilon)^{-1}r$ be such that $\openball M(a,s)\subseteq V$.
Fix any $b\in\openball M(a,s)$. Since $\varphi$ is \lowerptlip{(\alpha+\varepsilon)}
at $a\in M$ for every $\varepsilon>0$, the set $\psi(b)$ is non-empty,
and hence there exists some $z\in\psi(b)$. But $z\in\parenth{y+(\alpha+\varepsilon)d(a,b)\closedball X}\subseteq y+r\openball X\subseteq U$.
Therefore $\psi(b)\cap U\neq\emptyset$ for all $b\in\openball M(a,s)$.
We conclude that $\psi$ is lower hemicontinuous at $a$.

Let $c\in M\backslash\{a\}$ be arbitrary. We claim that $\psi$ is
lower hemicontinuous at $c$. Let $U\subseteq X$ be an open set satisfying
$\psi(c)\cap U\neq\emptyset$. Since $\varphi$ is \lowerptlip{(\alpha+\varepsilon)}
at $a\in M$ for every $\varepsilon>0$, by Lemma~\ref{lem:convex-intersection-withball-nonempty},
the set $\varphi(c)\cap\parenth{y+(\alpha+\varepsilon)d(a,c)\openball X}\cap U$
is non-empty. Let
\[
z\in\varphi(c)\cap\parenth{y+(\alpha+\varepsilon)d(a,c)\openball X}\cap U,
\]
 and choose $r>0$ satisfying both $0<r<(\alpha+\varepsilon)d(a,c)-\norm{y-z}$
and
\[
z+r\openball X\subseteq\parenth{y+(\alpha+\varepsilon)d(a,c)\openball X}\cap U.
\]

Since $\varphi$ is lower hemicontinuous, there exists some neighborhood
$V\subseteq M$ of $c$, such that $b\in V$ implies $\varphi(b)\cap(z+r\openball X)\neq\emptyset$.
By definition of $r$, we have $(\alpha+\varepsilon)d(a,c)-\norm{y-z}-r>0$,
hence we choose $s>0$ to satisfy both
\[
0<s<d(a,c)-(\alpha+\varepsilon)^{-1}\norm{z-y}-(\alpha+\varepsilon)^{-1}r
\]
 and $\openball M(c,s)\subseteq V$. Then, for $b\in\openball M(c,s)$,
by the reverse triangle inequality,
\begin{eqnarray*}
(\alpha+\varepsilon)d(a,b) & \geq & (\alpha+\varepsilon)d(a,c)-(\alpha+\varepsilon)d(b,c)\\
 & > & (\alpha+\varepsilon)d(a,c)-(\alpha+\varepsilon)s\\
 & > & (\alpha+\varepsilon)d(a,c)-(\alpha+\varepsilon)d(a,c)+\norm{z-y}+r\\
 & = & \norm{z-y}+r.
\end{eqnarray*}
Hence, for $b\in\openball M(c,s)$ and any $w\in z+r\openball X$,
\begin{eqnarray*}
\norm{w-y} & \leq & \norm{w-z}+\norm{z-y}\\
 & < & r+\norm{z-y}\\
 & < & (\alpha+\varepsilon)d(a,b),
\end{eqnarray*}
so that $(z+r\openball X)\subseteq\parenth{y+(\alpha+\varepsilon)d(a,b)\openball X}$
for all $b\in\openball M(c,s)$. By definition of $V$, for any $b\in\openball M(c,s)\subseteq V$,
there exists some $z'\in(z+r\openball X)\cap\varphi(b)\neq\emptyset$.
But then $z'\in(z+r\openball X)\subseteq U$ and
\[
z'\in(z+r\openball X)\subseteq\parenth{y+(\alpha+\varepsilon)d(a,b)\openball X}\subseteq\parenth{y+(\alpha+\varepsilon)d(a,b)\closedball X},
\]
so that $z'\in\psi(b)\cap U$. Hence $\psi(b)\cap U\neq\emptyset$
for every $b\in\openball M(c,s)$, and therefore $\psi$ is lower
hemicontinuous at $c$.

We finally conclude that $\psi$ is lower hemicontinuous.
\end{proof}
\begin{prop}
\label{prop:lower-(alpha+eps)-Lip-implies-admits-strong-pw-lip-selections}Let
$(M,d)$ be a metric space, $X$ a Banach space and $\alpha\geq0$.
Let $\varphi:M\tocorr X$ be a closed\textendash{} and convex-valued
lower hemicontinuous correspondence that is \lowerptlip{(\alpha+\varepsilon)}
for every $\varepsilon>0$. Then, for every $\varepsilon>0$, the
correspondence $\varphi$ admits strongly pointwise $(\alpha+\varepsilon)$-Lipschitz
selections.
\end{prop}

\begin{proof}
Let $\varepsilon>0$, $b\in M$ and $y\in\varphi(b)$ be arbitrary.
By Proposition~\ref{prop:lower-pt-lip-implies-lower-hemicontinuous-of-psi},
the correspondence $\psi:M\tocorr X$, defined by
\[
\psi(a):=\varphi(a)\cap(y+(\alpha+\varepsilon)d(b,a)\closedball X)\quad(a\in M),
\]
is lower hemicontinuous, as well as being closed\textendash , convex\textendash ,
and non-empty-valued. Since $M$ is paracompact by Stone's Theorem
(Theorem~\ref{thm:stone's-theorem}), by applying Michael's Selection
Theorem (Theorem~\ref{thm:michael-selection-theorem}), there exists
a continuous selection $f:M\to X$ of $\psi$. By definition of $\psi$,
the function $f$ is also a continuous selection of $\varphi$. Furthermore,
$f$ is strongly pointwise $(\alpha+\varepsilon)$-Lipschitz at $b\in M$
and satisfies $f(b)=y$.

Since $\varepsilon>0$, $b\in M$ and $y\in\varphi(b)$ were chosen
arbitrarily, we conclude that $\varphi$ admits strongly pointwise
$(\alpha+\varepsilon)$-Lipschitz selections for every $\varepsilon>0$.
\end{proof}
Easy applications of Proposition~\ref{prop:lower-(alpha+eps)-Lip-implies-admits-strong-pw-lip-selections}
and our Pointwise Lipschitz Selection Theorem (Theorem~\ref{thm:ptwise-lip-selection-theorem})
yield the following two corollaries. Compared to Theorem~\ref{thm:ptwise-lip-selection-theorem},
these two corollaries give more natural (but less general) sufficient
conditions on a correspondence for the existence of a continuous selection
that is pointwise Lipschitz on a dense set of its domain.
\begin{cor}
\label{cor:natural-ptwise-lip-sel-thm1}Let $(M,d)$ be a metric space,
$X$ a Banach space and $\alpha\geq0$. Let $\varphi:M\tocorr X$
be a closed\textendash{} and convex-valued lower hemicontinuous correspondence
that is \lowerptlip{(\alpha+\varepsilon)} for every $\varepsilon>0$.
Then, for any $\beta>\alpha$, there exists a continuous selection
of $\varphi$ that is pointwise $\beta$-Lipschitz on a dense set
of $M$.

If, additionally, there exists a bounded continuous selection of $\varphi$,
then, for any $\beta>\alpha$, there also exists a bounded continuous
selection of $\varphi$ that is pointwise $\beta$-Lipschitz on a
dense set of $M$.
\end{cor}

\begin{proof}
Since $\varphi$ is \lowerptlip{(\alpha+\varepsilon)}, it is also
non-empty valued. By Michael's Selection Theorem (Theorem~\ref{thm:michael-selection-theorem})
$\varphi$ has a continuous selection. Let $\beta>\alpha$ and define
$\varepsilon:=2^{-1}(\beta-\alpha)$. By Proposition~\ref{prop:lower-(alpha+eps)-Lip-implies-admits-strong-pw-lip-selections},
the correspondence $\varphi$ admits strongly pointwise $(\alpha+\varepsilon)$-Lipschitz
selections. We note that $\alpha+\varepsilon<\alpha+2\varepsilon=\beta$.
Then, by Theorem~\ref{thm:ptwise-lip-selection-theorem}, there exists
a continuous selection of $\varphi$ that is pointwise $\beta$-Lipschitz
on a dense set of $M$. Furthermore, if $\varphi$ has a bounded continuous
selection, Theorem~\ref{thm:ptwise-lip-selection-theorem} ensures
the existence of a bounded continuous selection of $\varphi$ that
is pointwise $\beta$-Lipschitz on a dense set of $M$.
\end{proof}
{}
\begin{cor}
\label{cor:natural-ptwise-lip-sel-thm2}Let $(M,d)$ be a \textbf{bounded}
metric space, $X$ a Banach space and $\alpha\geq0$. Let $\varphi:M\tocorr X$
be a closed\textendash{} and convex-valued lower hemicontinuous correspondence
that is \lowerptlip{(\alpha+\varepsilon)} for every $\varepsilon>0$.
Then, for any $\beta>\alpha$, there exists a bounded continuous selection
of $\varphi$ that is pointwise $\beta$-Lipschitz on a dense set
of $M$.
\end{cor}

\begin{proof}
Let $\varepsilon>0$ be arbitrary. By Proposition~\ref{prop:lower-(alpha+eps)-Lip-implies-admits-strong-pw-lip-selections},
the correspondence $\varphi$ admits strongly pointwise $(\alpha+\varepsilon)$-Lipschitz
selections. I.e., for any $b\in M$ and $y\in\varphi(b)$, there exists
a continuous selection $f$ of $\varphi$ that is strongly pointwise
$(\alpha+\varepsilon)$-Lipschitz selection at $b$. Since $M$ is
bounded, this selection $f$ is a bounded continuous selection of
$\varphi.$ Applying Corollary~\ref{cor:natural-ptwise-lip-sel-thm1}
yields the result.
\end{proof}

\section{Application: An improved Bartle-Graves Theorem\label{sec:Applications}}

We recall the following version of the classical Bartle-Graves Theorem:
\begin{thm}[{Classical Bartle-Graves Theorem \cite[Corollary~17.67]{AliprantisBorder}}]
\label{thm:classical-Bartle-Graves}Let $X$ and $Y$ be Banach spaces.
Every continuous linear surjection $T:X\to Y$ has a continuous (and
positively homogeneous) right inverse $\tau:Y\to X$ (i.e., $T\circ\tau=\text{id}_{Y}$).
\end{thm}

In this section we leverage our Pointwise Lipschitz Selection Theorem
(Theorem~\ref{thm:ptwise-lip-selection-theorem}) to establish a
slight improvement of the classical Bartle-Graves Theorem. Since the
case where $Y$ is finite dimensional is trivial (because the kernel
of $T$ is then complemented), we restrict ourselves to the infinite
dimensional case. In Theorem~\ref{thm:improved-Bartle-Graves} we
show that the map $\tau$ in the above theorem can, in fact, be chosen
to be pointwise Lipschitz on a dense set of $Y$. The construction
of this dense set, through application of Theorem~\ref{thm:ptwise-lip-selection-theorem},
yields a meager set.

The proof of Theorem~\ref{thm:improved-Bartle-Graves} is essentially
a straightforward verification of the lower pointwise Lipschitz-ness
of the inverse image correspondence (restricted to the unit sphere
of the codomain). This allows us to apply Corollary~\ref{cor:natural-ptwise-lip-sel-thm2}
to complete the proof.
\begin{thm}[Improved Bartle-Graves Theorem]
\label{thm:improved-Bartle-Graves}Let $X$ and $Y$ be infinite
dimensional Banach spaces and $T:X\to Y$ a continuous linear surjection.
Then there exist a constant $\eta>0$ and a positively homogeneous
continuous right inverse $\tau:Y\to X$ of $T$ that is pointwise
$\eta$-Lipschitz on a dense \textbf{meager} set of $Y$.
\end{thm}

\begin{proof}
By the Open Mapping Theorem, there exists some $\gamma>0$ such that
$\gamma\openball Y\subseteq T(\openball X)$. We define the correspondence
$\varphi:\sphere Y\tocorr X$ by $\varphi(y):=T^{-1}\{y\}$ for $y\in\sphere Y$.
It is clear that $\varphi$ is non-empty\textendash , closed\textendash ,
and convex\textendash valued.

We claim that $\varphi$ is lower hemicontinuous. Let $y\in\sphere Y$
be arbitrary and $U\subseteq X$ an open set satisfying $\varphi(y)\cap U\neq\emptyset$.
Let $x\in\varphi(y)\cap U$ be arbitrary and let $r>0$ be such that
$x+r\openball X\subseteq U$. Let $z\in\sphere Y\cap(y+r\gamma\openball Y)$
be arbitrary. Define $z':=z-y$ so that $z'\in r\gamma\openball Y$.
Then there exists some $x'\in r\openball X$ such that $Tx'=z'$,
and hence $T(x'+x)=z'+y=z-y+y=z$. Therefore $x+x'\in\varphi(z)\cap(x+r\openball X)$,
so that, for any $z\in\sphere Y\cap(y+r\gamma\openball Y)$, we have
$\varphi(z)\cap U\neq\emptyset.$ We conclude that $\varphi$ is lower
hemicontinuous.

Let $\varepsilon>0$ be arbitrary and set $\alpha:=\gamma^{-1}$.
We claim that $\varphi$ is lower pointwise $(\alpha+\varepsilon)$-Lipschitz.
Let $y\in\sphere Y$ and $x\in\varphi(y)$ be arbitrary. For any $z\in\sphere Y$,
define $z':=z-y$. Then $z'\in(1+\varepsilon\gamma)\norm{z-y}\openball Y$.
Let $x'\in\gamma^{-1}(1+\varepsilon\gamma)\norm{z-y}\openball X$
be such that $Tx'=z'$. Then $T(x+x')=y+z'=y+z-y=z,$ so that $x+x'\in\varphi(z)$.
But $x'\in(\alpha+\varepsilon)\norm{z-y}\openball X\subseteq(\alpha+\varepsilon)\norm{z-y}\closedball X$.
Hence $\varphi(z)\cap(x+(\alpha+\varepsilon)\norm{z-y}\closedball X)\neq\emptyset$,
and we conclude that $\varphi$ is lower pointwise $(\alpha+\varepsilon)$-Lipschitz
for every $\varepsilon>0$.

Let $\beta>\alpha$. By Corollary~\ref{cor:natural-ptwise-lip-sel-thm2},
there exists a selection $\underline{\tau}\in C(\sphere Y,X)$ of
the correspondence $\varphi$ that is pointwise $\beta$-Lipschitz
on a dense set of $\sphere Y$. We denote this dense set by $B\subseteq\sphere Y$,
which, by construction is meager (see the proof of Theorem~\ref{thm:ptwise-lip-selection-theorem}
where $B$ is constructed as $\bigcup_{n\in\N}B_{n}$, with $B_{n}$
being a $2^{-(n-1)}$-separation for each $n\in\N$. Since $\sphere Y$
was assumed to not be discrete, the set $B_{n}$ is nowhere dense
in $\sphere Y$ for each $n\in\N$).

It is straightforward to see that $B':=\set{\lambda b}{\lambda>0,\ b\in B}$
is dense and meager in $Y$. By Lemma~\ref{lem:homogeneous-exstension-preserves-pw-lip},
the positively homogeneous extension $\tau:Y\to X$ of $\underline{\tau}$
is pointwise $(2\beta+\norm{\underline{\tau}}_{\infty})$-Lipschitz
on $B'$. Setting $\eta:=2\beta+\norm{\underline{\tau}}_{\infty}$
and noting that $\tau$ is a right inverse of $T$ completes the proof.
\end{proof}

\ifx\ams\undefined\else\vfill\fi

\section{\label{sec:aharoni-lindenstrauss-example}An example of Aharoni and
Lindenstrauss}

The following example, devised by Aharoni and Lindenstrauss \cite{LindenstraussAharoni},
shows that continuous linear surjections between Banach spaces need
not have Lipschitz or even uniformly continuous right inverses in
general.
\begin{example}
\label{exa:cadlag-quotient}Let $D$ denote the space of all bounded
real-valued functions on $[0,1]$ that are right continuous at every
point of $[0,1]$, whose left limit exists at every point of $[0,1]$,
and with discontinuities only occurring at rational numbers. We endow
$D$ with the uniform norm $\norm{\cdot}_{\infty}$. Let $C([0,1])\subseteq D$
denote the closed subspace of all continuous real-valued functions
on $[0,1]$. The quotient map $q:D\to D/C([0,1])$ has no Lipschitz
(even uniformly continuous) right inverse. We refer the reader to
\cite{LindenstraussAharoni} or \cite[Example~1.20]{LindenstraussBenyamini}
for proof of this fact.
\end{example}

Our improved Bartle-Graves Theorem (Theorem~\ref{thm:improved-Bartle-Graves})
shows that the quotient map $q$ has a continuous positively homogeneous
right inverse that is pointwise $\eta$-Lipschitz for some $\eta>0$
on a dense meager set of $D/C([0,1])$. That the quotient map $q$
does not admit a Lipschitz right inverse shows that our Pointwise
Lipschitz Selection Theorem (Theorem~\ref{thm:ptwise-lip-selection-theorem})
cannot be improved to a general result which may yield a selection
that is pointwise Lipschitz on the \emph{whole }of its domain:

We first quote the following result by Sch\"affer \cite[Theorem~3.6]{Schaffer}:
\begin{thm}
\label{thm:unit-shperes-are-quasi-convex}The unit sphere of every
normed space is bi-Lipschitz homeomorphic to a length space (cf. \cite[Definition~2.1.6]{BuragoBuragoIvanov}).
\end{thm}

Next, a straightforward adaptation of a result due to Durand-Cartagena
and Jaramillo \cite[Corollary~2.4]{Durand-CartagenaJaramillo} yields
the following result:
\begin{thm}
\label{thm:durant-cartagena-jaramillo}Let $X$ be a normed space.
If a metric space $(M,d)$ is bi-Lipschitz homeomorphic to a length
space, then every function $f:M\to X$ that is pointwise $\alpha$-Lipschitz
for some $\alpha\geq0$ on the whole of $M$ is, in fact, Lipschitz.
\end{thm}

Returning to Example~\ref{exa:cadlag-quotient}, with the correspondence
$\varphi:\sphere{D/C([0,1])}\tocorr D$ defined by $\varphi(x):=q^{-1}(\{x\})$
for all $x\in\sphere{D/C([0,1])}$, should there exist a selection
of $\varphi$ that is pointwise $\alpha$-Lipschitz on the whole of
$\sphere{D/C([0,1])}$, we could be able to conclude that such a selection
is  Lipschitz by Theorems~\ref{thm:unit-shperes-are-quasi-convex}
and~\ref{thm:durant-cartagena-jaramillo}. The positively homogeneous
extension $f$ would then be a Lipschitz right inverse of $q$, contradicting
Aharoni and Lindenstrauss' observation that no such map exists.

\begin{acknowledgement*}
	The author would like to thank the MathOverflow community (Nate Eldredge
in particular, for pointing out the example in Remark~\ref{rem:cantor-ternary}
to the author), and the anonymous referees of the paper for their
constructive comments and suggestions.

\end{acknowledgement*}

\bibliographystyle{amsplain}
\bibliography{bib}

\end{document}